\documentclass[reqno]{amsart}
\usepackage{amsmath,amsthm,amssymb}
\usepackage{latexsym}
\usepackage{eucal}
\usepackage{tikz}

\usepackage{mathtools}
\usepackage[english]{babel}
\usepackage[autostyle, english = american]{csquotes}

\newtheorem{theorem}{Theorem}[section]
\newtheorem{lemma}{Lemma}[section]

\usepackage{amsmath,amsthm,amssymb}
\usepackage{latexsym}
\usepackage{eucal}
\usepackage{amsmath,amsthm,amssymb}
\usepackage{amsmath}
\usepackage{amsfonts}
\usepackage{amssymb}
\usepackage{mathtools}
\usepackage{amsthm}
\usepackage{latexsym}
\usepackage{eucal}
\usepackage{fullpage}
\usepackage{pgf}

\def\cal{\mathcal}
\let\Re=\undefined
\DeclareMathOperator{\Re}{Re}
\let\Im=\undefined
\DeclareMathOperator{\Im}{Im}

\def\ge{\geqslant}\def\le{\leqslant}

\def\~{\widetilde}
\def\cal{\mathcal}
\let\Re=\undefined
\DeclareMathOperator{\Re}{Re}
\let\Im=\undefined
\DeclareMathOperator{\Im}{Im}

\def\ge{\geqslant}
\def\le{\leqslant}
\def\~{\widetilde}

\begin{document}
\title[On Schur parameters in Steklov's problem
 ]{On Schur parameters in Steklov's problem}
\author{ S. Denisov, K. Rush }

\address{
\begin{flushleft}
Sergey Denisov: denissov@wisc.edu\\\vspace{0.1cm}
University of Wisconsin--Madison\\  Mathematics Department\\
480 Lincoln Dr., Madison, WI, 53706,
USA\vspace{0.1cm}\\and\\\vspace{0.1cm}
Keldysh Institute for Applied Mathematics, Russian Academy of Sciences\\
Miusskaya pl. 4, 125047 Moscow, RUSSIA\\
\vspace{1cm} Keith Rush: jkrush@wisc.edu
\\ \vspace{0.1cm}
University of Wisconsin--Madison\\  Mathematics Department\\
480 Lincoln Dr., Madison, WI, 53706, USA
\end{flushleft}
}\maketitle

\begin{abstract}
We study the recursion (aka Schur) parameters for monic polynomials
orthogonal on the unit circle with respect to a weight which
provides negative answer to the conjecture of Steklov.
\end{abstract}
\section{Introduction}

Given a probability measure $d\sigma$ on the unit circle $\mathbb{T}
=\partial \mathbb{D}, \mathbb{D}=\{z \in \mathbb{C}: |z| < 1\}$ with
infinitely many growth points, we define the monic polynomials
$\{\Phi_n\}$ orthogonal in $L^2(d\sigma)$ by beginning with the set
$\{1, z, z^2, \dots\}$ and applying the Gram-Schmidt
orthogonalization procedure, i.e.,
\[
\Phi_n(z,\sigma)=z^n+a_{n-1,n}z^{n-1}+\ldots+a_{0,n}, \quad
\int_{\mathbb{T}}
\Phi_n(e^{i\theta},\sigma)e^{-ij\theta}d\sigma(\theta)=0, \quad
j=0,\ldots,n-1.
\]
Consider the so-called $\ast$-operation (of order $n$) defined as
$$P_n^*(z) = z^n\overline{P_n(\bar{z}^{-1})}, \quad z\in \mathbb{C},$$ which gives
\[
P_n^*(z) = z^n\overline{P_n(z)}, \quad z\in \mathbb{T}
\]
for any polynomial $P_n$ of degree at most $n$. Given $\{\Phi_n\}$,
one can define the orthonormal polynomials by the formula
\begin{equation}\label{c1}
\phi_n(z,\sigma)=\frac{\Phi_n(z,\sigma)}{\|\Phi_n\|_{2,\sigma}},
\quad
\|\Phi_n\|_{2,\sigma}=\left(\int_{\mathbb{T}}|\Phi_n|^2d\sigma\right)^{1/2}.
\end{equation}
The polynomials of the second kind will be denoted by $\{\Psi_n\}$
(monic) and $\Psi_n^*(z) = z^n\overline{\Psi_n(\bar{z}^{-1})}, \,\,
z\in \mathbb{C}$, $\{\psi_n\}$ (orthonormal), and $\{\psi_n^*\}$.
The pair $(\Phi_n,\Phi_n^*)$ satisfies the Szeg\H{o} recurrence:
\begin{equation}\label{recur}
\left\{\begin{array}{cc} \Phi_{n+1}(z,\sigma) =&
z\Phi_n(z,\sigma)-\overline{\gamma}_n\Phi_n^*(z,\sigma)\\
\Phi_{n+1}^*(z,\sigma) =& \Phi_{n}^*(z,\sigma) - {\gamma}_n
z\Phi_n(z,\sigma)
\end{array}\right.
\end{equation}
The pair $(\Psi_j,\Psi_j^*)$ satisfies the same recurrence except
that the parameters are $\{-\gamma_n\}$.
 We
refer the reader to \cite{6,sim1} for the basic theory (our
$\gamma_j=\alpha_j$ from \cite{sim1}). The recursion parameters
$\{\gamma_n\} \subseteq \mathbb{D}^\infty$ are sometimes called the
Schur parameters due to their relationship with Schur functions and
the Schur algorithm. Two of the key identities in the theory are
\[
\|\Phi_n\|_{2,\sigma}=\left(\prod\limits_{j=0}^{n-1}
(1-|\gamma_j|^2)\right)^{1/2},
\]
\begin{equation}\label{sum-rule}
\exp\left(\frac1{4\pi}\int^\pi_{-\pi}\ln
(2\pi\sigma'(\theta))d\theta\right)=\prod_{j\,\ge0}\rho_j, \quad
\rho_j=(1-|\gamma_j|^2)^{1/2}\,\,.
\end{equation}
and the measure $\sigma$ is said to satisfy the Szeg\H{o} condition
if the left hand side in the last formula is positive (which is
equivalent to $\{\gamma_j\}\in \ell^2(\mathbb{Z}^+)$). Due to
\eqref{c1}, we have
\begin{equation}\label{nm}
|\phi_n(z)|\sim |\Phi_n(z)|
\end{equation}
uniformly in $n$ and $z$ provided that $\sigma$ is Szeg\H{o}
measure.

The condition on the measure to be probability controls only the
size of the orthonormal polynomials. Indeed,
\begin{equation}\label{ra1}
\Phi_j(z,\alpha \sigma)=\Phi_j(z,\sigma), \quad \phi_j(z,\alpha
\sigma)=\alpha^{-1/2} \phi_j(z,\sigma).
\end{equation}
In 1921, Steklov conjectured in \cite{stek} that, for a positive
weight $p$ given on the interval $[a,b]$ of the real line
$\mathbb{R}$, the corresponding sequence of orthonormal polynomials
$\{P_n(x,p)\}$ is bounded in $n$ for every fixed $x\in (a,b)$. This
conjecture attracted considerable interest (see, e.g.,
\cite{5}-\cite{Gol}) and was answered negatively by Rakhmanov in the
series of papers \cite{3,4}. The proofs first dealt with polynomials
orthogonal on the unit circle and then the obtained results were
recast to handle $\{P_n\}$ by making use of a formula due to
Geronimus. In this paper, we will only be focusing on the case of
orthogonality on the unit circle. We define the Steklov class of
measures as
\[
S_{\delta}=\left\{\sigma: \int d\sigma=1,\quad \sigma'\ge
\delta/(2\pi), \quad{\rm a.e.\quad} \theta\in [0,2\pi)\right\},
\]
where $\delta\in (0,1)$. The following variational problem was
considered in \cite{3,4}
\[
M_{n,\delta}=\sup_{\sigma\in
S_\delta}\|\phi_n(z,\sigma)\|_{L^\infty(\mathbb{T})}.
\]
In \cite{4}, the  estimates
\[
\left(\frac{n}{\ln ^3 n}\right)^{1/2}<_\delta  M_{n,\delta}<_\delta
n^{1/2}
\]
were established (see formula \eqref{aps} below for the definition
of $<_{\delta}$) and \cite{adt} improved it to
\[
M_{n,\delta}\sim_\delta n^{1/2}.
\]
The paper \cite{adt} contained a method that presents both the
measure ${\sigma^\ast}^{(n)}\in S_\delta$ and the polynomial
$\phi_n(z,{\sigma^\ast}^{(n)})$, which satisfies $
\|\phi_n(z,{\sigma^\ast}^{(n)})\|_\infty\sim_\delta n^{1/2} $,
explicitly. However, the parameters $\{\gamma_j^{(n)}\}$ that
correspond to this construction were defined only implicitly.
Neither did methods of Rakhmanov provide much information about the
behavior of $\{\gamma_j^{(n)}\}$. In the current paper, we study the
recursion parameters that give a negative answer to the conjecture
of Steklov. That opens new, difference-equation perspective to the
problem  and answers partially an open question raised in
\cite{adt}. The main result of the current paper is

\begin{theorem}\label{th1}
Fix $\epsilon\in (0,\epsilon_0]$ where $\epsilon_0$ is sufficiently
small. Then, there is $n_0(\epsilon)$ such that for every
$n>n_0(\epsilon)$, there is a weight ${\tilde{w}}^{(n)}$ such that
\begin{equation}\label{raz}
{\tilde{w}}^{(n)} {\it \,satisfies\,the\,uniform\,Steklov\, condition:\quad}
\left\|\frac{1}{{\tilde{w}}^{(n)}}\right\|_{L^\infty(\mathbb{T})}\lesssim
1,
\end{equation}
\begin{equation}\label{dva}
\|\phi_{2n+1}(z,\tilde{w}^{(n)})\|_{L^\infty(\mathbb{T})}>_\epsilon
\ln n,
\end{equation}
and the asymptotics for $\gamma_j^{(n)}$ is given by
\begin{equation}\label{qq1}
   \gamma_j^{(n)} = -\left\{
     \begin{array}{l}
       \displaystyle \frac{i^{j+1}}{j+1}
\sum_{\sigma=\pm 1} \sigma^{j+1} (2(j+1))^{\frac{ 2i\epsilon\sigma
}{\pi}} \frac{\Gamma(1-\frac{i\epsilon\sigma}{\pi})}
       {\Gamma\left(\frac{i\epsilon\sigma}{\pi}
       \right)} +r_{j,\epsilon},\quad 0\leq j \leq n-1\\
       -\displaystyle \frac{i^{j'+1}}{j'+1}

\sum_{\sigma=\pm 1} \sigma^{j'+1} (2(j'+1))^{\frac{ 2i\epsilon\sigma
}{\pi}} \frac{\Gamma(1-\frac{i\epsilon\sigma}{\pi})}
       {\Gamma\left(\frac{i\epsilon\sigma}{\pi}
       \right)}
 +
        r_{j',\epsilon},\quad
         n \leq j \leq 2n-1\\
       0, \hspace{3cm} j = 2n \\
       \displaystyle \frac{(-1)^{j}i^{j-2n}}{ j-2n      }

\sum_{\sigma=\pm 1} \sigma^{j-2n} (2(j-2n))^{\frac{ 2i\epsilon\sigma
}{\pi}} \frac{\Gamma(1-\frac{i\epsilon\sigma}{\pi})}
       {\Gamma\left(\frac{i\epsilon\sigma}{\pi}
       \right)}+r_{j-2n,\epsilon},\quad
        2n+1\leq j \leq 3n\\
       0, \hspace{3cm} j > 3n
 \end{array}
   \right.
\end{equation}
where $j'=2n-1-j$ and $|r_{j,\epsilon}|<C_\epsilon (j+1)^{-2}$.
\end{theorem}

{\bf Remark.} The existence of the measure satisfying these
conditions is not a new result \cite{murman} and the logarithmic
growth is not optimal \cite{dn}. However, the construction and the
asymptotics of the Schur parameters are new and might be
interesting. The polynomial constructed by our method has a
structure similar to the one from \cite{3,d3}.

{\bf Remark.} The careful analysis of \eqref{qq1} shows that the
main terms are real-valued and converge to $0$ as $\epsilon\to 0$
for every fixed $j$. The results in \cite{deift} allow one to
control the dependence of $C_\epsilon$ on $\epsilon$ when it
converges to zero and we conjecture that $\lim_{\epsilon\to
0}{C_\epsilon}=0$.
\smallskip

We will use the following notation. If $f_{1(2)}(x)$ are two
positive functions for which $ f_1<Cf_2 $ with some absolute
constant $C$, uniformly in the argument, we will write $f_1\lesssim
f_2$. If $f_1\lesssim f_2$ and $f_2\lesssim f_1$, then $f_1\sim
f_2$. If
\begin{equation}\label{aps}
\sup_{x}\frac{f_1}{f_2}<C(\epsilon),
\end{equation}
where $\epsilon$ is a parameter, then we will write
$f_1<_{\epsilon}f_2$. Relations $f_1\sim _\epsilon f_2,
f_1>_\epsilon f_2$ are defined similarly.

 The
symbol $\Gamma$ denotes the Gamma function. We will call a function
$F$ to be Carath\'eodory if it is analytic in $\mathbb{D}$ and its
real part is positive. Given a measure $\sigma$, we will denote the
Carath\'eodory function given by the Schwarz transform of $\sigma$
as
\[
F(z)=\mathcal{S}(\sigma)=\int_{-\pi}^\pi
\frac{e^{i\theta}+z}{e^{i\theta}-z}d\sigma(\theta).
\]
If $\sigma$ is in Szeg\H{o} class, $\Pi$ will denote the outer
function in $\mathbb{D}$ that satisfies
\[
|\Pi(z)|^{-2}=2\pi\sigma', \quad {\rm a.e.}\, z\in \mathbb{T}, \quad
\Pi(0)>0\,.
\]

\smallskip

The structure of the paper is as follows: the second section
contains the proofs of auxiliary Lemmas and the main Theorem. The
third section is an Appendix with the proof of a Lemma in the main
text.

\section{Proof of Theorem \ref{th1}}

The measures considered below will be symmetric with respect to
$\mathbb{R}$ and the related Schur parameters will be real. We will
need the following simple Lemma.
\begin{lemma}\label{l2}
Suppose $\Phi_k, \Psi_k, \Phi_k^*, \Psi_k^*$ are the polynomials
that correspond to {\bf real} Schur parameters $
\{\alpha_j\}_{j=0}^{k-1}$. Then, the polynomials associated to the
sequence of Schur parameters $\{\gamma_j\}$ given by
$$\gamma_j = \left\{
     \begin{array}{lr}
       \alpha_j & : 0 \leq j\leq k-1\\
       -\alpha_{2k-1-j} & : k \leq j \leq 2k-1
     \end{array}
   \right.$$
satisfy
$$2\Phi_{2k} = \Phi_k^2+\Phi_k\Psi_k-z^{-1}(\Phi_k^*)^2+z^{-1}\Phi_k^*\Psi_k^*,$$
$$2\Phi_{2k}^* = (\Phi_k^*)^2+\Phi_k^*\Psi_k^*-z\Phi_k^2+z\Phi_k\Psi_k.$$

\end{lemma}

\begin{proof}
It is known that the pair $(\Psi_j,\Psi_j^*)$ satisfies the
Szeg\H{o} recurrence with  parameters $\{-\alpha_j\}$. If
\begin{equation}\label{aa1}
 A = \prod_{j=k-1}^{0}\begin{bmatrix}
   z       & -\gamma_j \\
    -z\gamma_j       & 1 \\
\end{bmatrix}
=
\begin{bmatrix}
               a      & b \\
                c      & d \\
            \end{bmatrix},
 \end{equation}
 then
 \[
\begin{bmatrix}
               \Phi_k     & \Psi_k \\
                \Phi_k^*      & -\Psi_k^* \\
            \end{bmatrix}=A\begin{bmatrix}
               1      & 1 \\
                1      & -1 \\
            \end{bmatrix}.
 \]
Thus,
$$a = \frac{\Phi_{k}+\Psi_{k}}{2},\,
b = \frac{\Phi_{k}-\Psi_{k}}{2},\,
c=\frac{\Phi_{k}^*-\Psi_{k}^*}{2},\,
d=\frac{\Phi_{k}^*+\Psi_{k}^*}{2}.$$ First we reverse the dynamics.
We are interested in
$$ \prod_{j=0}^{k-1}\begin{bmatrix}
   z       & -\gamma_j \\
    -z\gamma_j       & 1 \\

\end{bmatrix}\,.
$$
We see that
$$A^T = \prod_{j=0}^{k-1}
\begin{bmatrix}
   z      & -z\gamma_j \\
    -\gamma_j       & 1 \\
\end{bmatrix}
 = \prod_{j=0}^{k-1} \left(
 \begin{bmatrix}
    1      & 0 \\
     0      & z^{-1} \\
 \end{bmatrix}
 \begin{bmatrix}
    z      & -\gamma_j \\
     -z\gamma_j       & 1 \\
 \end{bmatrix}
 \begin{bmatrix}
     1      & 0 \\
      0      & z \\
  \end{bmatrix}
  \right)
  =
  \begin{bmatrix}
      1      & 0 \\
     0      & z^{-1} \\
   \end{bmatrix}
   \prod_{j=0}^{k-1}
   \left(  \begin{bmatrix}
       z      & -\gamma_j \\
        -z\gamma_j       & 1 \\
    \end{bmatrix} \right)
    \begin{bmatrix}
         1      & 0 \\
          0      & z \\
      \end{bmatrix}\,.
  $$
  Therefore,
  $$ \prod_{j=0}^{k-1}\begin{bmatrix}
     z       & -\gamma_j \\
      -z\gamma_j       & 1 \\

  \end{bmatrix}
   =
  \begin{bmatrix}
        1      & 0 \\
       0      & z \\
     \end{bmatrix}
     A^T
     \begin{bmatrix}
           1      & 0 \\
          0      & z^{-1} \\
        \end{bmatrix}\,.
        $$
We have
$$
\begin{bmatrix}
    1      & 0 \\
     0      & -1 \\
 \end{bmatrix}
 \begin{bmatrix}
    z      & -\gamma_j \\
     -z\gamma_j       & 1 \\
 \end{bmatrix}
 \begin{bmatrix}
     1      & 0 \\
      0      & -1 \\
  \end{bmatrix} =  \begin{bmatrix}
      z      & \gamma_j \\
      z\gamma_j       & 1 \\
   \end{bmatrix}\,.
  $$
  Therefore, \eqref{aa1} implies
  $$\prod_{j=k-1}^0 \begin{bmatrix}
      z      & \gamma_j \\
       z\gamma_j       & 1 \\
   \end{bmatrix}
   =
   \begin{bmatrix}
       1      & 0 \\
        0      & -1 \\
    \end{bmatrix}
    A
    \begin{bmatrix}
        1      & 0 \\
         0      & -1 \\
     \end{bmatrix}\,.
     $$
Combining the above results, we get
     $$\prod_{j=0}^{k-1} \begin{bmatrix}
           z      & \gamma_j \\
            z\gamma_j       & 1 \\
        \end{bmatrix}
        =
\begin{bmatrix}
                1      & 0 \\
                 0      & -z \\
             \end{bmatrix}
         A^T
         \begin{bmatrix}
                        1      & 0 \\
                         0      & -z^{-1} \\
                     \end{bmatrix}\,,
                     $$
and so
$$
 \begin{bmatrix}
                \Phi_{2k}      & \Psi_{2k} \\
                 \Phi_{2k}^*      & -\Psi_{2k}^*\\
   \end{bmatrix}
             =
 \begin{bmatrix}
 1      & 0 \\
0      & -z \\
 \end{bmatrix}
 \begin{bmatrix}
 a      & c \\
 b      & d \\
 \end{bmatrix}
\begin{bmatrix}
1      & 0 \\
0      & -z^{-1} \\
\end{bmatrix}
\begin{bmatrix}
a      & b \\
 c     & d \\
\end{bmatrix}
\begin{bmatrix}
 1      & 1 \\
 1      & -1 \\
 \end{bmatrix}
$$

$$=  \begin{bmatrix}
                a(a+b)-z^{-1}c(c+d)      & a(a-b)+z^{-1}c(d-c) \\
                 d(c+d)-zb(a+b)     & d(c-d)+zb(b-a) \\
             \end{bmatrix}\,.
             $$
Thus,
$$2\Phi_{2k} = \Phi_{k}^2+\Phi_{k}\Psi_{k} - z^{-1}(\Phi_{k}^*)^2 + z^{-1}\Phi_{k}^*\Psi_{k}^*$$
and
$$2\Phi_{2k}^* = (\Phi_{k}^*)^2 + \Phi_{k}^*\Psi_{k}^* - z\Phi_{k}^2+z\Phi_{k}\Psi_{k}.$$
\end{proof}

The following result is well-known.

\begin{lemma}\label{rot} If $\{\alpha_j\}$ are the Schur parameters for the
measure $\sigma(\theta)$, then parameters $\{\alpha_j^{(\beta)}\}$
for the translated measure $\sigma^{(\beta)}=\sigma(\theta-\beta)$
are given by
\[
\alpha_j^{(\beta)}=e^{-i(j+1)\beta}\alpha_j, \quad j=0,1,\ldots
\]
\end{lemma}
\begin{proof} The proof follows from making the following two
observations:
\[
\Phi_n(z,\sigma^{(\beta)})=e^{in\beta}\Phi_n(ze^{-i\beta})
\]
(which follows from the definition) and, take $z=0$ in
\eqref{recur},
$\Phi_{j+1}(0,\sigma^{(\beta)})=-\overline{\alpha}_j^{(\beta)}$.
\end{proof}
 We will need the following  decoupling Lemma. Its proof
is contained in \cite{adt,d3}. However, following \cite{d3}, we
include the sketch for the reader's convenience.

\begin{lemma}\label{decop}
Suppose we are given a  polynomial $\phi_n$ of degree $n$ and
Carath\'eodory function $\~F$
 which satisfy the following properties
\begin{itemize}
\item[1.] $\phi_n^*(z)$ has no roots in $\overline{\mathbb D}$.
\item[2.] Normalization on the size and ``rotation\!"
\begin{equation}\label{norma}
\int_\mathbb T|\phi_n^*(z)|^{-2}d\theta
=2\pi~,\quad\phi_n^*(0)>0\,\,.
\end{equation}

\item[3.] $\~F\!\in\!C^\infty(\mathbb T)$, $\Re\~F>0$ on $\mathbb T$, and
\begin{equation}\label{norka}
\frac{1}{2\pi}\int_\mathbb T\Re\~F(e^{i\theta})d\theta=1\,\,.
\end{equation}

\end{itemize}
Denote the Schur parameters given by the probability measures
$\mu_n$ and $\~\sigma$
\[
d\mu_n=\frac{d\theta}{2\pi|\phi_n^*(e^{i\theta})|^2}, \quad
d\~\sigma=\~\sigma'd\theta=\frac{\Re \~F(e^{i\theta})}{2\pi}d\theta,
\]
as $\{\gamma_j\}$ and $\{\~\gamma_j\}$, respectively. Then, the
probability measure $\sigma$, corresponding to Schur coefficients
\[
\gamma_0,\ldots, \gamma_{n-1},\~\gamma_0,\~\gamma_1,\ldots
\]
is purely absolutely continuous with the weight given by
\begin{equation}\label{mp}
\sigma'=\frac{4\~\sigma'}{|\phi_n+\phi_n^*+\~F(\phi_n^*-\phi_n)|^2}=\frac{2\Re\~F}
{\pi|\phi_n+\phi_n^*+\~F(\phi_n^*-\phi_n)|^2}\,\,.
\end{equation}
The polynomial $\phi_n$ is the $n-$th orthonormal polynomial for $\sigma$.
\end{lemma}

\begin{proof}
First, notice that $\{\~\gamma_j\}\in \ell^1$ by Baxter's Theorem
(see, e.g., \cite{sim1}, Vol.1, Chapter 5). Therefore, $\sigma$ is
purely absolutely continuous by the same Baxter's criterion. Define
the orthonormal polynomials of the first/second kind corresponding
to measure $\~\sigma$ by $\{\~\phi_j\}, \{\~\psi_j\}$. Similarly,
let $\{\phi_j\}, \{\psi_j\}$ be orthonormal polynomials for
$\sigma$. Since, by construction, $\mu_n$ and $\sigma$ have
identical first $n$ Schur parameters, $\phi_n$ is $n$-th orthonormal
polynomial for~$\sigma$.

 Let
us compute the polynomials $\phi_j$ and $\psi_j$, orthonormal with
respect to $\sigma$, for the indexes $j>n$. The recursion can be
rewritten in the following matrix form
\begin{equation}\label{m-ca}
\left(\begin{array}{cc}
\phi_{n+m} & \psi_{n+m}\\
\phi_{n+m}^* & -\psi_{n+m}^*
\end{array}\right)=\left(\begin{array}{cc}
{\cal A}_m & {\cal B}_m\\
{\cal C}_m & {\cal D}_m
\end{array}\right)\left(\begin{array}{cc}
\phi_{n} & \psi_{n}\\
\phi_{n}^* & -\psi_{n}^*
\end{array}\right)\,,\end{equation}
where ${\cal A}_m, {\cal B}_m, {\cal C}_m, {\cal D}_m$ satisfy
\begin{eqnarray*}\left(\begin{array}{cc}
{\cal A}_0 & {\cal B}_0\\
{\cal C}_0 & {\cal D}_0
\end{array}\right)=\left(\begin{array}{cc}
1 & 0\\
0 & 1
\end{array}\right),\hspace{6cm}\\
\left(\begin{array}{cc}
{\cal A}_m & {\cal B}_m\\
{\cal C}_m & {\cal D}_m
\end{array}\right)=\frac1{\~\rho_0\cdot\ldots\cdot\~\rho_{m-1}}\left(\begin{array}{cc}
z & -\overline{\~\gamma}_{m-1}\\
-z\~\gamma_{m-1} & 1
\end{array}\right)\cdot\ldots\cdot\left(\begin{array}{cc}
z & -\overline{\~\gamma}_0\\
-z\~\gamma_0 & 1
\end{array}\right)\end{eqnarray*}
and thus depend only on $\~\gamma_0,\ldots,\~\gamma_{m-1}$.
Moreover, we have
\[\left(\begin{array}{cc}
\~\phi_m & \~\psi_m\\
\~\phi_m^* & -\~\psi^*_m
\end{array}\right)=\left(\begin{array}{cc}
{\cal A}_m & {\cal B}_m\\
{\cal C}_m & {\cal D}_m
\end{array}\right)\left(\begin{array}{cc}
1 & 1\\
1 & -1
\end{array}\right)
\,\,.
\]
Thus, ${\cal A}_m\!=\!(\~\phi_m\,{+\,\~\psi_m)/2,~{\cal
B}_m\!=\!(\~\phi_m\,-}\,\~\psi_m)/2,~ {\cal
C}_m\!=\!(\~\phi^*_m\,{-\,\~\psi^*_m)/2,~{\cal
D}_m\!=\!(\~\phi^*_m\,+}\,\~\psi^*_m)/2$ and their substitution into
\eqref{m-ca} yields
\begin{equation}\label{intert}
2\phi_{n+m}^*=\phi_n(\~\phi_m^*-\~\psi^*_m)+\phi_n^*(\~\phi_m^*+\~\psi^*_m)=
\~\phi_m^*\left(\phi_n+\phi_n^*+\~
F_m(\phi_n^*-\phi_n)\right)\end{equation} where
\[
\~F_m(z)=\frac{\~\psi^*_m(z)}{\~\phi^*_m(z)}\,\,.
\]
Since $\{\~\gamma_n\}\!\in\!\ell^1$ and $\{\gamma_n\}\!\in\!\ell^1$,
we have (\cite{sim1}, p.~225)
\[
\~F_m\to\~F~{\rm as~}m\to\infty~{\rm and~}
\phi_j^*\to\Pi,~\~\phi_j^*\to\~\Pi~{\rm as~}j\to\infty\,\,.
\]
uniformly on $\overline{\mathbb D}$. The functions $\Pi$ and $\~\Pi$
are related to $\sigma$ and $\~\sigma$ as follows: they are the
outer functions in $\mathbb D$ that satisfy
\begin{equation}\label{facti}
|\Pi|^{-2}=2\pi\sigma',\quad |\~\Pi|^{-2}=2\pi\~\sigma',\quad
\Pi(0)>0,\quad \~\Pi(0)>0
\end{equation}
on $\mathbb{T}$. In \eqref{intert}, send $m\to\infty$ to get
\begin{equation}\label{facti1}
2\Pi=\~\Pi\left(\phi_n+\phi_n^*+\~F(\phi_n^*-\phi_n)\right)
\end{equation}
and we have \eqref{mp} after taking the square of absolute values
and using \eqref{facti}.
\end{proof}

The next result has to do with the recent analysis of the
asymptotics of the polynomials orthogonal with the respect to the
so-called Fisher-Hartwig weights. We will use \cite{deift} as the
main reference.\smallskip

Consider the weight on $\mathbb{T}$ given by

\begin{equation}\label{ves}
f(z) = e^\epsilon g_{i, -\frac{i\epsilon}{\pi}}(z)g_{-i,
\frac{i\epsilon}{\pi}}(z),
\end{equation}
where $z = e^{i\theta}, \, \theta \in [0, 2\pi)$ and $\epsilon\in
(0,\epsilon_0]$ with $\epsilon_0$ to be chosen sufficiently small.
For $z_j = e^{i\theta_j}$,
\[
 g_{z_j, \beta_j} =  \left\{
     \begin{array}{lr}
       e^{i\pi\beta_j} & : 0\leq\arg z < \theta_j\\
       e^{-i\pi\beta_j} & : \theta_j \leq \arg z < 2\pi
     \end{array}
     \right..
\]
That is, $f$ is a weight with two jumps, one from $e^\epsilon$ to
$e^{-\epsilon}$ around $i$ (in counterclockwise direction), and one
from $e^{-\epsilon}$ back to $e^\epsilon$ around $-i$ (again in
counterclockwise direction). It does not define a probability
measure but this will not influence the polynomials much due to
\eqref{nm},\eqref{ra1}. Notice that $f$ is symmetric with respect to
$\mathbb{R}$ so all its recursion parameters are real.

We will need the following two Lemmas the proofs of which
(essentially contained in \cite{deift}) will be discussed in the
Appendix.

\begin{lemma}\label{l4}
The Schur parameters $\{\gamma_j\}$ associated to the $f$ above
satisfy:
$$\gamma_j =-(j+1)^{-1}i^{j+1}\left((2(j+1))^{\frac{2i\epsilon}{\pi}}
\frac{\Gamma(1-\frac{i\epsilon}{\pi})}{\Gamma\left(\frac{i\epsilon}{\pi}\right)}
+ (-1)^{j+1} (2(j+1))^{-\frac{2i\epsilon}{\pi}}
\frac{\Gamma\left(1+\frac{i\epsilon}{\pi}\right)}
{\Gamma\left(\frac{-i\epsilon}{\pi}\right)}\right)+r_{j,\epsilon},
\quad |r_{j,\epsilon}|<C_\epsilon (j+1)^{-2}\,.$$
\end{lemma}

The next Lemma will be needed in the proof of Theorem \ref{th1}.

\begin{lemma}\label{l1} Let $\epsilon\in (0,\epsilon_0]$ and $n>n_0(\epsilon)$. Then, for the weight $f$ given by \eqref{ves}, the $n$-th associated
monic polynomials of the first and second kinds $\Phi_n$ and $\Psi_n$ respectively, and their $*$-polynomials
 $\Psi_n^* $ and $\Phi_n^*$ satisfy the following estimates:

 \begin{equation}|\Phi_n^*(z)|\sim 1, \quad z \in \mathbb{T},\label{uno}
\end{equation}

\begin{equation}\|\Phi_n^*\Psi_n^*+z\Phi_n\Psi_n\|_{L^\infty(\mathbb{T})}>_{\epsilon}\ln
n\label{duo}, \quad z\in \mathbb{T}\,,
\end{equation}

\begin{equation}\label{trez}
\left|\frac{\Psi_n^*(z)}{\Phi_n^*(z)} +
\frac{\Psi_n^*(-z)}{\Phi_n^*(-z)}\right|\lesssim 1, \quad z \in
\mathbb{T}\,.
\end{equation}
\end{lemma}
{\bf Remark.} The following general identity is immediate from
Szeg\H{o} recursion by taking the determinant
\[
\det\left[
\begin{array}{cc}
\Phi_n & \Psi_n\\
\Phi_n^*& -\Psi_n^*
\end{array}
\right]=-2z^n\prod_{j=0}^{n-1}(1-|\gamma_j|^2)
\]
so
\[
\overline{\Phi}_n^*\Psi_n^*+\Phi_n^*\overline{\Psi}_n^*=2\prod_{j=0}^{n-1}(1-|\gamma_j|^2)\,.
\]
Then, \eqref{uno} gives
\begin{equation}\label{quatro}
\left|\frac{\Psi_n^*}{\Phi_n^*}+\overline{\left(\frac{\Psi_n^*}{\Phi_n^*}\right)}\right|\lesssim
1
\end{equation}
uniformly over $\mathbb{T}$ if the polynomials correspond to the
weight \eqref{ves}.\bigskip

\bigskip

Now we are in position to give a proof of Theorem \ref{th1}.

\begin{proof}{\it(of Theorem \ref{th1}).}
Given $\epsilon>0$ and large $n$, we consider $f$ defined by
\eqref{ves}. If the corresponding coefficients are denoted by
$\{\alpha_j\}$, we consider the weight $w^{(n)}$ given by Schur
parameters $\{\gamma_j^{(n)}\}$ defined as
\[
\gamma_j^{(n)}=\left\{
\begin{array}{cc}
\alpha_j, & j\leq n-1\\
-\alpha_{j'}, &j'=2n-1-j, \, n\leq j\leq 2n-1\\
0, & j=2n\\
(-1)^{j-2n}\alpha_{j-2n-1}, & 2n+1\leq j\leq 3n
\end{array}
\right.\,.
\]
Now, $\eqref{qq1}$ follows immediately from Lemma \ref{l4} and we
need to show \eqref{raz} and \eqref{dva}.

Let us prove \eqref{dva}. Notice first that $|\phi_j|\sim |\Phi_j|$
by \eqref{nm} and \eqref{ra1} so it is sufficient to consider
$\Phi_{2n+1}$. We apply Lemma \ref{l2} to say
$$ 2\Phi_{2n} = \Phi_n^2+\Phi_n\Psi_n-z^{-1}{\Phi_n^*}^2+z^{-1}\Phi_n^*\Psi_n^*\,,$$
$$2\Phi_{2n}^* = {\Phi_n^*}^2+\Phi_n^*\Psi_n^*-z\Phi_n^2+z\Phi_n\Psi_n\,.$$
Since the $\gamma_{2n}^{(n)}=0$, we get $\Phi_{2n+1} = z\Phi_{2n},
\Psi_{2n+1} = z\Psi_{2n}, \Phi_{2n+1}^* = \Phi_{2n}^*$,
 and $\Psi_{2n+1}^* = \Psi_{2n}^*$ so
\begin{equation}\label{idi} 2\Phi_{2n+1} =
z\Phi_n^2+z\Phi_n\Psi_n-{\Phi_n^*}^2+\Phi_n^*\Psi_n^*,\quad
2\Phi_{2n+1}^* =
{\Phi_n^*}^2+\Phi_n^*\Psi_n^*-z\Phi_n^2+z\Phi_n\Psi_n.\end{equation}
We have
\[
2\|\Phi_{2n+1}\|_{\infty}\geq
\|\Phi_n^*\Psi_n^*+z\Phi_n\Psi_n\|_\infty-2\|\Phi_n\|_{\infty}^2>_{\epsilon}\ln
n
\]
by Lemma \ref{l1}.

To show \eqref{raz}, we will use Lemma \ref{decop}. We choose
Carath\'eodory function for the decoupled problem as
 \begin{equation}\label{ees1}\widetilde{F}(z) =\frac{\psi_n^*(-z)}{\phi_n^*(-z)}=
\frac{\Psi_n^*(-z)}{\Phi_n^*(-z)}=-\frac{\Psi_n^*(z)}{\Phi_n^*(z)} +
O(1)\quad  \text{ by  \eqref{trez}}.\end{equation} We have (see
\cite{sim1}, Theorem 3.2.4)
$$\Re \widetilde{F}(e^{i\theta}) = |\phi_n^*(e^{i(\theta+\pi)})|^{-2}$$
and $\widetilde F$ is Carath\'eodory function of the
Bernstein-Szeg\H{o} weight
$(2\pi)^{-1}|\phi_n^*(e^{i(\theta+\pi)})|^{-2}$ having the Schur
parameters
\[
\left\{
\begin{array}{cc}
 (-1)^{j+1}\alpha_j, &j<n\\
 0, & j\geq n
 \end{array}\right.
 \] as follows from Lemma \ref{rot}.
Since our Schur parameters $\gamma_{j}^{(n)}=0, j>3n$, we have
\[
|\widetilde\Pi(-z)|=|\phi_n^*(-z,f)|\sim  1, \quad {\rm by\quad
\eqref{uno}}.
\]
So, by the identity \eqref{facti1}, we only need to show
$$|\Phi_{2n+1}+\Phi_{2n+1}^* + \widetilde{F}(\Phi_{2n+1}^*-\Phi_{2n+1})|\lesssim 1$$
uniformly on $\mathbb{T}$.  Recall that $\displaystyle \widetilde{F}
= -\frac{\Psi_n^*}{\Phi_n^*}+O(1)$ by \eqref{trez}. We introduce
auxiliary
\[D=\frac{\Phi_n^*\Psi_n^*}{2},\quad A = -\frac{{\Phi_n^*}^2}{2}.\] Notice
that
\begin{equation}\label{sept}
|A|\sim 1, \quad \frac{D}{A} =
-\frac{\Psi_{n}^*}{\Phi_n^*}=-\overline{\left(\frac{D}{A}\right)}+O(1)
\end{equation}
by \eqref{uno} and \eqref{quatro}. We can now rewrite \eqref{idi} as
\[
\Phi_{2n+1}=-A^{*}+D^*+A+D, \quad \Phi_{2n+1}^*=-A+D+A^*+D^*\,.
\]
(where the $(*)$-operations in these  identities are of order
$2n+1$).

 Then, by \eqref{ees1},
$$\Phi_{2n+1}+\Phi_{2n+1}^* + \widetilde{F}(\Phi_{2n+1}^*-\Phi_{2n+1}) =2\left((D+D^*)+\frac{D}{A}(A^*-A)\right)+O(1)=$$
$$
2\left(D^*+\frac{D}{A}A^*\right)+O(1)=\frac{2z^{2n}}{\overline{A}}\left(
\frac{D}{A}+\overline{\left(\frac{D}{A}\right)}  \right)=O(1)
$$
by \eqref{sept}.

\end{proof}
\bigskip

\section{Appendix: properties of polynomials and Schur parameters, \\proofs of Lemma \ref{l1} and Lemma \ref{l4}}
\subsection{Setup}
We wish to analyze the asymptotics of the orthogonal polynomials
with respect to the weight $$f(z) = e^\epsilon g_{i,
-\frac{i\epsilon}{\pi}}(z)g_{-i, \frac{i\epsilon}{\pi}}(z),$$ where
$z = |z|e^{i\theta}, \theta \in [0, 2\pi)$. For $z_j =
e^{i\theta_j}$,
\[ g_{z_j, \beta_j} =  \left\{
     \begin{array}{lr}
       e^{i\pi\beta_j} & : 0\leq \arg z < \theta_j\\
       e^{-i\pi\beta_j} & : \theta_j \leq \arg z < 2\pi
     \end{array}
     \right.
\]
as defined in the main text.

This $f$ belongs to the class of Fisher-Hartwig weights considered
in \cite{deift} (see also \cite{am} for the weight on the real
line), so much of this
 Appendix consists of examining the results of \cite{deift}, which performs asymptotic analysis of a
 Riemann-Hilbert problem involving the orthogonal polynomials. Recall the
  three properties of these polynomials we need:

\begin{enumerate}
\item $$|\Phi_n^*(z)| \sim 1,$$
\item$$\|\Phi_n^*\Psi_n^* + z\Phi_n\Psi_n\|_{L^\infty(\mathbb{T})} >_\epsilon \ln n,$$
 \item$$\frac{\Psi_n^*(z)}{\Phi_n^*(z)}+\frac{\Psi_n^*(-z)}{\Phi_n^*(-z)} = O(1),$$
 \end{enumerate}
for $z\in \mathbb{T}, n>n_0(\epsilon)$.

As was noted in \cite{baik} and \cite{fokas}, the orthogonal polynomials of the first and second kinds satisfy a particular
 Riemann-Hilbert problem with contour $C = \mathbb{T}$. If $Y$ is
 defined by
\begin{equation} \label{rhp} Y(z) = \left( \begin{array}{cc}
\Phi_n(z) & \displaystyle \int_\mathbb{T}
\frac{\Phi_n(\xi)}{\xi-z}\frac{f(\xi)d\xi}{2\pi i\xi^n} \\&\\
-\Phi_{n-1}^*(z) & \displaystyle
-\int_\mathbb{T}\frac{\overline{\Phi_{n-1}}(\xi)}{\xi-z}\frac{f(\xi)d\xi}{2\pi
i\xi}  \\ \end{array} \right),
\end{equation}
then it satisfies the following Riemann-Hilbert problem:
\begin{itemize}
\item $Y(z)$ is analytic in $\mathbb{C} \backslash \mathbb{T}$.
\item For $z \in \mathbb{T} \backslash{\{i, -i\}}$, $Y$ has continuous boundary values $Y_+(z)$ as $z$ approaches $\mathbb{T}$
 from the inside, and $Y_-(z)$ from the outside, related by the jump condition
$$Y_+(z) = Y_-(z)\left( \begin{array}{cc}
1 & z^{-n}f(z) \\
0 & 1  \\ \end{array} \right).$$

\item $Y(z)$ has the following asymptotic behavior at infinity:
$$Y(z)=\left(I+O\left(\frac{1}{z}\right)\right)\left( \begin{array}{cc}
z^n & 0 \\
0 & z^{-n}  \\ \end{array} \right).$$
\item As $z \to \pm i, z \in \mathbb{C}\backslash \mathbb{T}$,
$$Y(z) = \left( \begin{array}{cc}
O(1) & O(\ln|z- \pm i|) \\
O(1) &O(\ln|z-\pm i|)  \\ \end{array} \right).$$
\end{itemize}
In what follows, we consider $n$  to be a sufficiently large but
fixed parameter.

The analysis of Riemann-Hilbert problems of this type  proceeds by:
 enclosing the singularity points $z_j$ by small but fixed disks $U_{z_j}$;
  enclosing the curve $C = \mathbb{T}$ by lenses meeting in these disks; solving the
   Riemann-Hilbert problems induced in these regions; finally stitching them back together.
    For Fisher-Hartwig singularities this was performed in \cite{deift}.
     Since two of our estimates must be uniform over $z \in \mathbb{T}$, we will be concerned with the formulas
      that result in the regions enclosed by $U_{z_j}$ as well as the the regions outside $U_{z_j}$ but inside the
      lenses.

Before we begin the analysis of Riemann-Hilbert problem, we list a
few useful identities and notations. Following \cite{deift},
our complex logarithm will be cut at the negative real axis unless
otherwise noted.

The measure $d\mu=\displaystyle \frac{f}{2\pi\cosh \epsilon}d\theta$
is a probability one. Therefore, one has (see \cite{sim1}, equation
(3.2.53)):
$$\mathcal{S}(\mu)\Phi_{n-1}^*(z)-\Psi_{n-1}^*(z) = z^{n-1}\int_{\mathbb{T}} \frac{e^{i\theta}+z}{e^{i\theta}-z}
\overline{\Phi_{n-1}(e^{i\theta})} d\mu$$
$$=z^{n-1}\int_{\mathbb{T}} \frac{e^{i\theta}-z+2z}{e^{i\theta}-z} \overline{\Phi_{n-1}(e^{i\theta})} d\mu =
0+2z^n\int_\mathbb{T}\frac{\overline{\Phi_{n-1}}(\xi)}{\xi-z}\frac{f(\xi)d\xi}{
i2\pi \xi \cosh\epsilon }$$ so
\begin{equation} \label{y_22psi_n}
-Y_{22}(z) =
\int_\mathbb{T}\frac{\overline{\Phi_{n-1}}(\xi)}{\xi-z}\frac{f(\xi)d\xi}{2\pi
i\xi} = \frac{1}{2z^{n}}\left(F(z)\Phi_{n-1}^*(z)-(\cosh\epsilon)
\Psi_{n-1}^*(z)\right),
\end{equation}
where $F$ is the Carath\'{e}odory function associated
 to $f(\theta)\frac{d\theta}{2\pi}$ (i.e. $F = \mathcal{S}(f/2\pi)$), and $\Psi_n$ is the
  second kind polynomial associated to $\Phi_n$.

The exact expression for $F$ is easy to compute:
\begin{equation}
\label{Herglotz} F(z) =\int_{\mathbb{T}}
\frac{e^{i\theta}+z}{e^{i\theta}-z} f(\theta) \frac{d\theta}{2\pi}$$
$$= -\frac{i(e^\epsilon-e^{-\epsilon})}{\pi}   \ln
\left(\frac{i-z}{i+z}\right)+\frac{1}{2}(e^{-\epsilon}+e^\epsilon).
\end{equation}

 We will use the following notation from   \cite{deift}:
$$\beta_1 = -i\epsilon/\pi, \quad \alpha_1 = 0, \quad z_1 = i$$
$$\beta_2 = i\epsilon/\pi, \quad \alpha_2 = 0, \quad z_2 = -i$$

Before we discuss the results obtained in \cite{deift}, we need to
introduce confluent hypergeometric function  $\psi(a, b, \zeta)$
which plays the key role in the analysis of the Riemann-Hilbert
problem. We will have to use many facts about $\psi$. For this
purpose we refer the reader to the National Institute of Standards
and Technology's Digital Library of Mathematical Functions
\cite{dlmf} and the appendix of \cite{its}.

The function $\psi(a, b, \zeta)$ is the confluent hypergeometric
function of the second kind, often written as $U(a,b,\zeta)$. It is
defined as the unique solution to Kummer's equation
$$\zeta \frac{d^2w}{d\zeta^2} + (b-\zeta)\frac{dw}{d\zeta}-aw = 0,$$
satisfying $w(a, b, \zeta) \sim \zeta^{-a}$ as $\zeta\to \infty$. We
will be interested in the following choices of the parameters: $b=1$
and $a=\{\beta_j, 1+ \beta_j\}, j=1,2$. The function $\psi$ is
analytic in $\zeta$ on the universal cover of $\mathbb{C}\backslash
0$ and can be represented by the series (formula 13.2.9 in
\cite{dlmf}):
\begin{equation}\label{series}
\psi(a, 1, \zeta) = -\frac{1}{\Gamma(a)} \sum_{k=0}^\infty
\frac{(a)_k}{k!^2}\zeta^k\left(\ln\zeta +
\frac{\Gamma'(a+k)}{\Gamma(a+k)}-\frac{2\Gamma'(k+1)}{\Gamma(k+1)}\right),
\end{equation}
where $(a)_k = \displaystyle \frac{\Gamma(a+k)}{\Gamma(a)}$ is the
Pochhammer symbol. This allows us to write $\psi$ as
\begin{equation}\label{sht}
\psi(\zeta)=g(\zeta)\ln \zeta+h(\zeta), \end{equation} where $g$ and
$h$ are entire and single-valued. In particular, we have
\begin{equation}
\label{smallzeta} \psi(a, 1, \zeta) =
-\frac{1}{\Gamma(a)}\left(\ln\zeta +
\frac{\Gamma'(a)}{\Gamma(a)}-\frac{2\Gamma'(1)}{\Gamma(1)}\right) +
O(\zeta\ln\zeta)
\end{equation}
for $\zeta: |\zeta|<1$. The precise asymptotics of $\psi$ as
$\zeta\in \mathbb{C} \to \infty, -3\pi/2< \arg \zeta < 3\pi/2 $ for
fixed $a$ is (formula (7.2) in \cite{its})
\begin{equation} \label{largezeta} \psi(a, 1, \zeta) = \zeta^{-a}[1-a^2\zeta^{-1}+O(\zeta^{-2})]
\end{equation}
This asymptotics is a consequence of the following integral
representation of $\psi$ (formula (7.3), \cite{its}):
\[
\psi(a,1,\zeta)=\frac{1}{\Gamma(a)}\int_0^{\infty e^{-i\alpha}}
t^{a-1}(1+t)^{-a}e^{-\zeta t}dt, \quad -\pi<\alpha<\pi,
\,-\pi/2+\alpha<\arg \zeta<\pi/2+\alpha.
\]
That representation, in particular, implies that
\begin{equation}\label{ss1}
\sup_{u\in i\mathbb{R},|u|>1}|\psi(\pm i\epsilon,1,u)|\to 1, \quad
\epsilon\to 0
\end{equation}
and
\begin{equation}\label{ss2}
\sup_{u\in i\mathbb{R},|u|>1}|\psi(1\pm
i\epsilon,1,u)-\psi(1,1,u)|\to 0, \quad \epsilon\to 0.
\end{equation}
The crucial property of $\psi$ which makes it indispensable for the
Riemann-Hilbert analysis is the following transformation formula
(formula (7.30) in \cite{its})
\[
\psi(a,c,e^{-2\pi i}\zeta)=e^{2\pi i a}\psi(a,c,\zeta)-\frac{2\pi
i}{\Gamma(a)\Gamma(a-c+1)}e^{i\pi a}e^\zeta
\psi(c-a,c,e^{-i\pi}\zeta).
\]
Following \cite{deift}, we will use the convention that, unless
otherwise mentioned, $\zeta$ always satisfies $ 0 \le \arg \zeta<
2\pi$.\smallskip

Concerning the logarithmic derivative of the Gamma function (the
digamma function) which appears above, we will have occasion to use
its reflection formula (equation 5.15.6, \cite{dlmf})
\begin{equation} \label{reflection}
\frac{\Gamma'(1-z)}{\Gamma(1-z)} - \frac{\Gamma'(z)}{\Gamma(z)} =
\pi \cot (\pi z).
\end{equation}\smallskip

Now, we need to discuss another function which will be important below. Consider
$$\mathcal{D}(z)= \exp\left(\frac{1}{2\pi i} \int_{\mathbb{T}} \frac{\ln f(s)}{s-z}
ds\right).$$ This function is analytic away from $\mathbb{T}$ with
a.e. boundary values satisfying $\mathcal{D}_+(z) =
\mathcal{D}_-(z)f(z), |\mathcal{D}_+(z)|^2=f(z)$ (compare with $\Pi$
defined in \eqref{facti}). Notice this definition a priori differs
from that in \cite{sim1}, section 2.4. However due to the
normalization $\int_{\mathbb{T}} \ln(f(\theta))d\theta = 0$, in fact
these two agree. We may compute $\mathcal{D}$ explicitly:
\begin{equation}
\label{Szego}
\mathcal{D}(z) = \exp\left(\frac{1}{2\pi i}
 \int_{\mathbb{T}} \frac{\ln f(s)}{s-z} ds\right) = \exp\left(\frac{\epsilon}{\pi i}
 \ln\left(\frac{i-z}{i+z}\right)\right).
\end{equation}
\bigskip

\subsection{Asymptotic formulas for solution of the Riemann-Hilbert problem}

In the next two subsections, we recall how asymptotics of $Y$ on
$\mathbb{T}$ is obtained through solving Riemann-Hilbert problem.
Then we will apply this asymptotics to prove Lemma \ref{l1}.

In \cite{deift}, the Riemann-Hilbert problem undergoes  various
transformations until it is in a form for
 which explicit solutions can be written. The singularity points $i$ and $-i$, along with the artificially introduced point $z=1$
  (though the analysis reduces to triviality here and we drop this case) are all enclosed by the small disks $U_i$, $U_{-i}$ of fixed radius $\delta > 0$.
   The remainder of the unit circle is enclosed in ``lenses" (see  Figure~1).

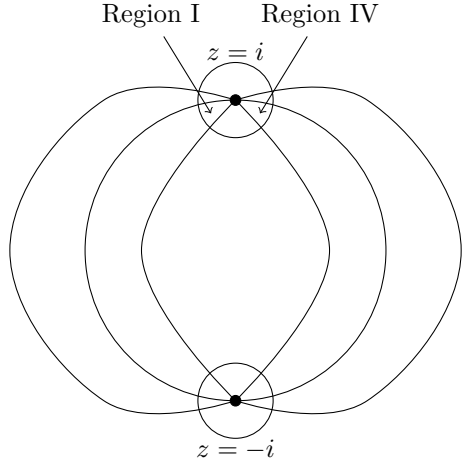
\begin{figure}
\centering

\begin{tikzpicture}

\draw  (0,0) node (v3) [anchor = north]{} ellipse (2 and 2);
\draw  (0,2) node (v1) {} ellipse (.5 and .5); \draw  (0,-2) node
(v2) {} ellipse (.5 and .5);

\draw  plot[smooth, tension=.7] coordinates {(0,2) (-1.25,0) (0,-2)
}; \draw  plot[smooth, tension=.7] coordinates {(0,2) (1.25,0)
(0,-2) }; \draw  plot[smooth, tension=.7] coordinates {(0,2)
(1.75,2) (3,0) (1.75,-2) (0,-2)}; \draw  plot[smooth, tension=.7]
coordinates {(0,2) (-1.75,2) (-3,0) (-1.75,-2) (0,-2)};
\filldraw[black] (0,2) circle (2pt);
\filldraw[black]  (0,-2) circle (2pt);

\node (va) at (0, -2.65) {$z=-i$};
\node (vb) at (0, 2.65) {$z=i$};

\node (v4) at (0,3.75) {};
%\draw  [->](v3) edge (v4);

\node (v5) at (3.75,0) {};

%\draw  [->](v3) edge (v5);
\node (v6) at (-1.125,3.125) {Region I}; \node (v7) at (-.25,1.7)
{}; \node (v8) at (1.125,3.125) {Region IV}; \node (v9) at (.25,1.7)
{}; \draw  [->](v6) edge (v7); \draw  [->](v8) edge (v9);
\end{tikzpicture}
\caption{Setup of Riemann-Hilbert problem on $\mathbb{T}$}
\end{figure}

We trace through the various transformations of the RH problem for
$z$ away from the points of singularity $i$ and $-i$. These
reductions are:

\begin{equation}\label{trans}
Y \to T \to S \to R\,.
\end{equation}
 We will explain each of these transformations
below. However, our analysis will be limited to the case when
$\epsilon\in (0,\epsilon_0]$,$n>n_0(\epsilon)$,\,$|z|\leq 1$ and $z$
 belongs to one of the lenses. This choice is motivated by our goal
 to control $Y$ only on the unit circle $\mathbb{T}$ itself so we will only need to take $|z|=1$ later on. In these domains, some of the transformations in \eqref{trans} are
trivial, e.g., $T=Y$ (formula (4.1) in \cite{deift}).

Then, (formula (4.3), \cite{deift}), $S$ is related to $T$ by
$$S(z) = T(z)\left( \begin{array}{cc}
1 & 0 \\
-f(z)^{-1}z^n & 1 \\\end{array} \right)\,. $$ Now, the original
Riemann-Hilbert problem for $Y$ can be written in terms of $S$ and
its solution proceeds by first choosing various parametrices
(approximate solutions) in each of the domains. The parametrices
outside of $\cup_j U_j$ and inside of each $U_j$ will be denoted by
$N$ and $P_{z_j}$, respectively. Our final transformation is to $R$,
which satisfies the  Riemann-Hilbert problem of very special form.
In fact, the correct choices of parametrices $N$ and $P_{z_j}$ makes
it possible to say that each jump in the Riemann-Hilbert problem for
$R$ is of the form $I+O(n^{-1})$ when $n\to\infty$ on each of the
contours involved (see (4.57)--(4.59)) and an asymptotics of $R$ at
infinity is $I+O(1/z)$. Then, the standard argument (see, e.g.,
\cite{dv1}) implies that
\begin{equation}R=I+O(n^{-1})\label{uni-r}\end{equation} uniformly
over $z\in \mathbb{C}$. It is clear now that the main asymptotics of
$Y$ is captured by $N$ and $P_{z_j}$. Below we will discuss these
parametrices in detail.

\subsubsection{Case 1. Parametrix $N$, $z$ outside of $U_{z_j}$}

For $z:|z|<1$, we write (formula (4.7), \cite{deift})
$$N(z) = \left( \begin{array}{cc}
\mathcal{D}(z) & 0 \\
0 & \mathcal{D}(z)^{-1} \\\end{array} \right)\left( \begin{array}{cc}
0 & 1 \\
-1 & 0 \\\end{array} \right)$$
$$R(z) = S(z)N^{-1}(z)$$
and $$R(z) = I+O\left(\frac{1}{n}\right)$$  by equations (4.61) and
(4.65-71) in \cite{deift}. Since we are away from the singularities
of the weight, all terms are uniformly bounded. Collecting
$O\left(\frac{1}{n}\right)$ errors, we have
$$S(z) = N(z) + O\left(\frac{1}{n}\right).$$
Reversing these transformations
$$Y(z) = T(z) = S(z) \left( \begin{array}{cc}
1 & 0 \\
-f(z)^{-1}z^n & 1 \\\end{array} \right)^{-1} = \left(N(z) + O\left(\frac{1}{n}\right)\right)\left( \begin{array}{cc}
1 & 0 \\
f(z)^{-1}z^n & 1 \\\end{array} \right) $$
$$= \left(\left( \begin{array}{cc}
\mathcal{D}(z) & 0 \\
 0 & \mathcal{D}(z)^{-1} \\\end{array} \right)\left( \begin{array}{cc}
 0 & 1 \\
-1 & 0 \\\end{array} \right) + O\left(\frac{1}{n}\right)\right)\left( \begin{array}{cc}
1 & 0 \\
f(z)^{-1}z^n & 1 \\\end{array} \right)\,.$$ So,
\begin{equation} \label{outsieUzj} Y(z) = \left( \begin{array}{cc}
z^nf^{-1}\mathcal{D}(z) & \mathcal{D}(z) \\
-\mathcal{D}(z)^{-1} & 0 \\\end{array} \right) +
O\left(\frac{1}{n}\right)\,,
\end{equation}
when $|z|<1$ and $z\notin U_{z_j}$. Now, to get asymptotical
behavior of $Y(z)$ (and thus of $\Phi_n^*(z)$ and $\Psi_n^*(z)$ by
formula \eqref{y_22psi_n}) on $\mathbb{T}$ but outside $U_{z_j}$, we
need to take $|z|\to 1$ in \eqref{outsieUzj} and this gives
\begin{equation} \label{outsieUzj1} Y(z) = \left( \begin{array}{cc}
z^nf^{-1}\mathcal{D}_+(z) & \mathcal{D}_+(z) \\
-\mathcal{D}_+(z)^{-1} & 0 \\\end{array} \right) +
O\left(\frac{1}{n}\right)
\end{equation}
uniformly over $z\in \mathbb{T}, z\notin U_{z_j}$.

\subsubsection{Case 2: Parametrix $P_{z_j}$, $z\in U_{z_j}$. }

The nature of the singularities of the weight $f$ at points $z_1$
and $z_2$ is the same so we will discuss only asymptotics in
$U_{z_1}$ in  detail. By formula (4.23) in \cite{deift}, we write
$$P_{z_j}(z) = E(z)\Psi^{(j)}(z)\left( \begin{array}{cc}
z^{n/2} & 0 \\
0 & z^{-n/2} \\\end{array} \right),$$ where  $E(z)$ is chosen so
that $P_{z_j}$ and $N$ approximately match across $\partial
U_{z_j}$. For example, we choose region~I on  Figure 1 (we can also
take the similar domain around $z_2$). This corresponds to $|z|<1,
0<\arg\frac{z}{z_j}<\pi/2$. By introducing $\zeta = n\ln
\frac{z}{z_j}$ we get $\arg \zeta\in (\pi/2,\pi)$ and the choice of
$E$ and $\Psi^{(j)}$ are made by (formula (4.32), \cite{deift})

$$\Psi_j^{(I)}(\zeta) = \left( \begin{array}{cc}
\psi(\beta_j, 1, \zeta)e^{2i\pi \beta_j}\left(\frac{z}{z_j}\right)^{-n/2} & -e^{i\pi \beta_j}\psi(1-\beta_j, 1, e^{-i\pi}\zeta)\left(\frac{z}{z_j}
\right)^{n/2}\frac{\Gamma(1-\beta_j)}{\Gamma(\beta_j)} \\
-e^{i\pi\beta_j}\psi(1+\beta_j, 1, \zeta)\left(\frac{z}{z_j}\right)^{-n/2}\frac{\Gamma(1+\beta_j)}{\Gamma(-\beta_j)} & \psi(-\beta_j, 1, e^{-i\pi}
\zeta)\left(\frac{z}{z_j}\right)^{n/2} \\\end{array} \right)$$
$$\coloneqq \left( \begin{array}{cc}
e^{2i\pi  \beta_j}\left(\frac{z}{z_j}\right)^{-n/2}\psi_1^{(j)} & -e^{i\pi  \beta_j}\left(\frac{z}{z_j}\right)^{n/2}c_2^{(j)}\psi_3^{(j)} \\
-e^{i\pi
\beta_j}\left(\frac{z}{z_j}\right)^{-n/2}c_1^{(j)}\psi_2^{(j)} &
\left(\frac{z}{z_j}\right)^{n/2}\psi_4^{(j)} \\\end{array} \right),
\quad ({\rm this \,\, defines\quad}
\psi^{(j)}_{1,2,3,4},c^{(j)}_{1,2})$$ and (formula (4.47))
$$E(z) = N(z)\left( \begin{array}{cc}
\zeta^{\beta_j} & 0 \\
0 & \zeta^{-\beta_j} \\\end{array} \right)\left( \begin{array}{cc}
z_j^{-n/2} & 0 \\
0 & z_j^{n/2} \\\end{array} \right)\left( \begin{array}{cc}
e^{-2\pi i \beta_j} & 0 \\
0 & e^{\pi i \beta_j} \\\end{array} \right)\,.$$ Multiplying
matrices, we find
$$P_{z_j}(z) = \left( \begin{array}{cc}
0 & \mathcal{D}(z) \\
-\mathcal{D}(z)^{-1} & 0 \\\end{array} \right) \left( \begin{array}{cc}
\zeta^{\beta_j}  \psi_1^{(j)} & -e^{-\pi i \beta_j}c_2^{(j)}\zeta^{\beta_j}z_j^{-n} \psi_3^{(j)} \\
-e^{2\pi i \beta_j}c_1^{(j)}z_j^{n}\zeta^{-\beta_j}\psi_2^{(j)} &
e^{\pi i \beta_j}\zeta^{-\beta_j}\psi_4^{(j)}
\\\end{array}\right)\,.$$
Restricting our attention to $j=1$, and using the notation
$Y^{(1, I)}$ to identify this as the solution $Y$ in region $I$
around point $z_1=i$:
$$Y^{(1, I)} = T = RP_{z_1}\left( \begin{array}{cc}
1 & 0 \\
f(z)^{-1}z^n & 1 \\\end{array} \right)\,.$$ We have again
   $$R(z) = S(z)P_{z_1}^{-1}(z)$$
   and
   $$R(z) = I+O\left(\frac{1}{n}\right)$$
   as follows from \eqref{uni-r}.  However, since we have  singularities in our expressions, we will leave $R$ in this form for now.
   This yields
\begin{equation}\label{abc}
Y^{(1, I)}(z) =\left(I+ O\left(\frac{1}{n}\right)\right) P_{z_1}(z)\left( \begin{array}{cc}
1 & 0 \\
f^{-1}z^n & 1 \\\end{array} \right)\end{equation}  $$= \left(I+
O\left(\frac{1}{n}\right)\right)\left( \begin{array}{cc}
0 & \mathcal{D}(z) \\
-\mathcal{D}(z)^{-1} & 0 \\\end{array} \right) \left( \begin{array}{cc}
\zeta^{-i\epsilon/\pi}  \psi_1^{(1)} & -e^{-\epsilon}c_2^{(1)}\zeta^{-i\epsilon/\pi}i^{-n} \psi_3^{(1)} \\
-e^{2\epsilon}c_1^{(1)}i^{n}\zeta^{i\epsilon/\pi}\psi_2^{(1)} & e^\epsilon\zeta^{i\epsilon/\pi}\psi_4^{(1)} \\\end{array}\right)\left( \begin{array}{cc}
1 & 0 \\
f^{-1}z^n & 1 \\\end{array} \right)$$
$$=\left(I+ O\left(\frac{1}{n}\right)\right)\left( \begin{array}{cc}
0 & \mathcal{D}(z) \\
-\mathcal{D}(z)^{-1} & 0 \\\end{array} \right) \left( \begin{array}{cc}
\zeta^{-i\epsilon/\pi}  \psi_1^{(1)} -e^{-\epsilon}f^{-1}z^nc_2^{(1)}\zeta^{-i\epsilon/\pi}i^{-n} \psi_3^{(1)}& -e^{-\epsilon}c_2^{(1)}\zeta^{-i\epsilon/\pi}i^{-n} \psi_3^{(1)} \\
-e^{2\epsilon}c_1^{(1)}i^{n}\zeta^{i\epsilon/\pi}\psi_2^{(1)} +e^\epsilon f^{-1}z^n\zeta^{i\epsilon/\pi}\psi_4^{(1)} & e^{\epsilon}\zeta^{i\epsilon/\pi}\psi_4^{(1)} \\\end{array}\right)$$
$$ = \left(I+ O\left(\frac{1}{n}\right)\right)\left( \begin{array}{cc}
\mathcal{D}(z)(-e^{2\epsilon}c_1^{(1)}i^{n}\zeta^{i\epsilon/\pi}\psi_2^{(1)} +e^{\epsilon}f^{-1}z^n\zeta^{i\epsilon/\pi}\psi_4^{(1)})& \mathcal{D}(z) e^\epsilon\zeta^{i\epsilon/\pi}\psi_4^{(1)} \\-\mathcal{D}(z)^{-1}(\zeta^{-i\epsilon/\pi}  \psi_1^{(1)} -e^{-\epsilon}f^{-1}z^nc_2^{(1)}\zeta^{-i\epsilon/\pi}i^{-n} \psi_3^{(1)})
 & \mathcal{D}(z)^{-1}e^{-\epsilon}c_2^{(1)}\zeta^{-i\epsilon/\pi}i^{-n} \psi_3^{(1)}
 \\\end{array}\right)\,.
$$
In $U_{z_2}$, the calculations are exactly the same, but with $z_1 =
i$ replaced by $z_2 = -i$ and $\beta_1$ replaced by $\beta_2 =
-\beta_1$.
 Therefore, we have
 \begin{eqnarray*}  Y^{(2, I)}(z)=\hspace{12cm} \\\left(I+ O\left(\frac{1}{n}\right)\right)\left( \begin{array}{cc}
 \mathcal{D}(z)(-e^{-2\epsilon}c_1^{(2)}(-i)^{n}\zeta^{-i\epsilon/\pi}\psi_2^{(2)}
 +e^{-\epsilon}f^{-1}z^n\zeta^{i\epsilon/\pi}\psi_4^{(2)})& \mathcal{D}(z)
 e^{-\epsilon}\zeta^{-i\epsilon/\pi}\psi_4^{(2)} \\-\mathcal{D}(z)^{-1}(\zeta^{i\epsilon/\pi}
  \psi_1^{(2)} -e^\epsilon f^{-1}z^nc_2^{(2)}\zeta^{i\epsilon/\pi}(-i)^{-n} \psi_3^{(2)})
  & \mathcal{D}(z)^{-1}e^\epsilon c_2^{(2)}\zeta^{i\epsilon/\pi}(-i)^{-n} \psi_3^{(2)}
  \\\end{array}\right).
  \end{eqnarray*}
Recalling that
 \[ f(z)= \left\{
       \begin{array}{lr}
         e^\epsilon, &  -\pi/2\leq\arg z < \pi/2\\
         e^{-\epsilon}, &  \text{otherwise}
       \end{array}
       \right.,
 \]
we see that $f = e^{-\epsilon}$ in $U_{z_1}$, region I,  and  $f =
e^\epsilon$  in region I of $U_{z_2}$. Thus, we get
$$Y^{(1, I)}(z) = \left(I+ O\left(\frac{1}{n}\right)\right)\left( \begin{array}{cc}
e^{2\epsilon}\mathcal{D}(z)(-c_1^{(1)}\zeta^{i\epsilon/\pi}i^n\psi_2^{(1)} +z^n\zeta^{i\epsilon/\pi}\psi_4^{(1)})& \mathcal{D}(z) e^\epsilon\zeta^{i\epsilon/\pi}\psi_4^{(1)} \\-\mathcal{D}(z)^{-1}(\zeta^{-i\epsilon/\pi}  \psi_1^{(1)}- z^nc_2^{(1)}\zeta^{-i\epsilon/\pi}i^{-n} \psi_3^{(1)})
 & \mathcal{D}(z)^{-1}e^{-\epsilon}c_2^{(1)}\zeta^{-i\epsilon/\pi}i^{-n} \psi_3^{(1)} \\\end{array}\right) $$
 and
  $$ Y^{(2, I)}(z)  =\left(I+ O\left(\frac{1}{n}\right)\right)\left( \begin{array}{cc}
  e^{-2\epsilon}\mathcal{D}(z)(-c_1^{(2)}\zeta^{-i\epsilon/\pi}(-i)^n\psi_2^{(2)} +z^n\zeta^{-i\epsilon/\pi}\psi_4^{(2)})& \mathcal{D}(z) e^{-\epsilon}\zeta^{-i\epsilon/\pi}\psi_4^{(2)} \\-\mathcal{D}(z)^{-1}(\zeta^{i\epsilon/\pi}  \psi_1^{(2)} -z^nc_2^{(2)}\zeta^{i\epsilon/\pi}(-i)^{-n} \psi_3^{(2)})
   & \mathcal{D}(z)^{-1}e^\epsilon c_2^{(2)}\zeta^{i\epsilon/\pi}(-i)^{-n} \psi_3^{(2)} \\\end{array}\right)\,.$$
Because of the singularities involved, we must take care in
performing this last multiplication.
   Denote by $\widetilde{Y}_{z_j}$ the right-hand matrix in the above equation. That is
   \begin{equation}
   \label{insideUz1}
\widetilde{Y}_i^{(I)}(z) \coloneqq \left( \begin{array}{cc}
e^{2\epsilon}\mathcal{D}(z)(-c_1^{(1)}\zeta^{i\epsilon/\pi}i^n\psi_2^{(1)} +z^n\zeta^{i\epsilon/\pi}\psi_4^{(1)})& \mathcal{D}(z) e^\epsilon\zeta^{i\epsilon/\pi}\psi_4^{(1)} \\-\mathcal{D}(z)^{-1}(\zeta^{-i\epsilon/\pi}  \psi_1^{(1)}- z^nc_2^{(1)}\zeta^{-i\epsilon/\pi}i^{-n} \psi_3^{(1)})
 & \mathcal{D}(z)^{-1}e^{-\epsilon}c_2^{(1)}\zeta^{-i\epsilon/\pi}i^{-n} \psi_3^{(1)}
 \\\end{array}\right)\,,
   \end{equation}
   \begin{equation}
   \label{insideUz2}
\widetilde{Y}_{-i}^{(I)}(z)\coloneqq  \left( \begin{array}{cc}
  e^{-2\epsilon}\mathcal{D}(z)(-c_1^{(2)}\zeta^{-i\epsilon/\pi}(-i)^n\psi_2^{(2)} +z^n\zeta^{-i\epsilon/\pi}\psi_4^{(2)})& \mathcal{D}(z) e^{-\epsilon}\zeta^{-i\epsilon/\pi}\psi_4^{(2)} \\-\mathcal{D}(z)^{-1}(\zeta^{i\epsilon/\pi}  \psi_1^{(2)} -z^nc_2^{(2)}\zeta^{i\epsilon/\pi}(-i)^{-n} \psi_3^{(2)})
   & \mathcal{D}(z)^{-1}e^\epsilon c_2^{(2)}\zeta^{i\epsilon/\pi}(-i)^{-n} \psi_3^{(2)}
   \\\end{array}\right)\,.
   \end{equation}

\vspace{0.5cm}

We start with proving Lemma \ref{l4}.

\begin{proof}{\it (of Lemma \ref{l4}).}
We only need to use formula (1.23) from \cite{deift}. This equation
shows, in the notations introduced above,
\begin{equation}\label{anek1}
-\overline{\gamma}_{k-1} = \Phi_{k}(0)
 =  k^{-2\beta_1-1}z_1^k 2^{2\beta_2} \frac{\Gamma(1+\beta_1)}{\Gamma(-\beta_1)}+ k^{-2\beta_2-1}z_2^k 2^{2\beta_1}
 \frac{\Gamma(1+\beta_2)}{\Gamma(-\beta_1)} +
 r_{k,\epsilon}\end{equation}
and
\begin{equation}\label{anek2}
|r_{k,\epsilon}|<C_\epsilon (k+1)^{-2}.
\end{equation}
\end{proof}

{\bf Remark.}  The estimates \eqref{anek1}, \eqref{anek2}, and
\eqref{sum-rule} imply that the recursion coefficients for the
weight $f$ satisfy
\begin{equation}\label{nn}
\|\{\gamma_k\}\|_{\ell^2}\lesssim \sqrt\epsilon, \quad
|\gamma_k|<_\epsilon (k+1)^{-1}.
\end{equation}
Recall that  $\Psi_n$ and $\Psi_n^*$ satisfy recursion
\begin{equation}\label{recur}
\left\{\begin{array}{cc} \Psi_{n+1} =&
z\Psi_n+\overline{\gamma}_n\Psi_n^*\\
\Psi_{n+1}^*=& \Psi_{n}^* + {\gamma}_n z\Psi_n
\end{array}\right.
\end{equation}
and we have
\[
|\Psi_{n+1}|\leq |\Psi_{n}|(1+|\gamma_{n}|), \quad |\Psi_0|=1
\]
Iterating this formula and using \eqref{nn} we get a rough upper
bound
\begin{equation}\label{rough}
\|\Psi_n\|_{L^\infty(\mathbb{T})}<_\epsilon n^{C_\epsilon}.
\end{equation}
This estimate can be substantially improved by Riemann-Hilbert
analysis but \eqref{rough}  will be good enough for our
purposes.\smallskip

Now we are ready to verify Lemma \ref{l1}.

\begin{proof}{\it (of Lemma \ref{l1}).}

\subsection{$|\Phi_n^*(z)| \sim 1, z \in \mathbb{T}$ for $\epsilon\in (0,\epsilon_0]$ and $n>n_0(\epsilon)$}\quad\\

We consider two cases.

\subsubsection{$z$ outside $U_{z_j}$}

By equation \eqref{outsieUzj}, $\Phi_n^*(z) \to
\mathcal{D}_+(z)^{-1}$ uniformly in this region. Since
$|\mathcal{D}_+| = f^{1/2}$ a.e. on $\mathbb{T}$, this trivially
implies our desired estimate for $z$ outside of $U_{\pm i}$.

\subsubsection{$z$ inside $U_{z_j}$}

We consider the boundary values as $|z|\to 1$ in the asymptotics for
$Y$. Notice that, since $|\mathcal{D}_+|^2 = f$, we have
$|\mathcal{D}_+| = e^{-\epsilon/2}$  in region I around $z=i$. We
will focus on region I where $\zeta=iu, u>0$. In the other regions,
analysis is the same. Recall that $\zeta = n \ln \frac{z}{i}$. So,
if $z=e^{i(\pi/2+\tau)}, \tau>0$, we have
$$\zeta^{-i\epsilon/\pi} = \left(n\ln \frac{z}{i}\right)^{-i\epsilon/\pi} = (ni\tau)^{-i\epsilon/\pi} =
 e^{\epsilon/2}e^{-\frac{i\epsilon}{\pi}\ln (n\tau)}.$$
Therefore, in region I, in which $\arg\frac{z}{i}> 0$, one has
\begin{equation}\label{eql}\left|\frac{1}{\mathcal{D}(z)\zeta^{i\epsilon/\pi}}\right| =
e^{\epsilon}.\end{equation} Similarly, in region IV, the other side
of $i$, $|\mathcal{D}(z)| = e^{\epsilon/2}$ and
 $|\zeta^{-i\epsilon/\pi}| =|e^{-i\epsilon \pi^{-1}\ln \zeta}|= e^{3\epsilon/2}$
 since $\arg\zeta=3\pi/2$). We again obtain $\displaystyle \left|\frac{1}{\mathcal{D}(z)\zeta^{i\epsilon/\pi}}\right| = e^{\epsilon}$.

Consider the expressions involving the $\psi$ in the first column of
$\widetilde{Y}_i^{(I)}(z)$.
 We focus on $\widetilde{Y}_{i, 21}^{(I)}$, the bottom left corner of the $\widetilde{Y}$ matrix.
Due to \eqref{eql} and definition of $\zeta$,
$$ \left|\widetilde{Y}_{i, 21}^{(I)}(z) \right| =e^{\epsilon} \left|\psi_1^{(1)}- \left(\frac{z}{i}\right)^n c_2^{(1)} \psi_3^{(1)}\right| =
e^\epsilon \left|\psi\left(-\frac{i\epsilon}{\pi} ,1, \zeta\right)-
e^\zeta\frac{\Gamma(1+\frac{i\epsilon}{\pi})}{\Gamma(-\frac{i\epsilon}{\pi})}\psi\left(1+\frac{i\epsilon}{\pi},
1, e^{-i\pi} \zeta\right)\right|.$$ Consider
\[
\Omega(\zeta,\epsilon)=\psi\left(-\frac{i\epsilon}{\pi} ,1,
\zeta\right)-
e^\zeta\frac{\Gamma(1+\frac{i\epsilon}{\pi})}{\Gamma(-\frac{i\epsilon}{\pi})}\psi\left(1+\frac{i\epsilon}{\pi},
1, e^{-i\pi} \zeta\right).
\]
It is the analysis of this function which concerns us. We want to
show that
 \[
 \max_{\zeta=iu, u\in \mathbb{R}}||\Omega(\zeta,\epsilon)|-1|\to 0, \quad
 \epsilon\to 0.
 \]
We do this in two steps.  The estimates
\eqref{largezeta},\eqref{ss1}, and \eqref{ss2} imply that
 \[
 \max_{\zeta=iu, |u|>1}||\Omega(\zeta,\epsilon)|-1|\to 0, \quad
 \epsilon\to 0.
 \]
For $|\zeta|<1$, we will use
 series \eqref{series} for $\psi$. We want to show
 \[
 \max_{\zeta=iu,|u|<1}||\Omega(\zeta,\epsilon)|-1|\to 0, \quad
 \epsilon\to 0.
 \]
Recall that $\Gamma(\zeta)$ has a pole at $0$ so $\lim_{\epsilon\to
0}\Gamma^{-1}(\pm i\epsilon)=0$. From \eqref{sht}, we get
\[
 \Omega(\zeta,\epsilon)=(\ln \zeta)
g(-i\epsilon/\pi,\zeta)+h(-i\epsilon/\pi,\zeta)-e^\zeta
\frac{\Gamma(1+i\epsilon/\pi)}{\Gamma(-i\epsilon/\pi)}\Bigl((\ln(e^{-i\pi}\zeta))g(1+i\epsilon/\pi,
e^{-i\pi}\zeta)+h(1+i\epsilon/\pi,e^{-i\pi}\zeta)\Bigr),
\]
where
\begin{equation}\label{ser1}
g(a,\zeta)=-\frac{1}{\Gamma(a)} \sum_{k=0}^\infty
\frac{(a)_k}{k!^2}\zeta^k,
\end{equation}
\begin{equation}\label{ser2}
h(a,\zeta)=-\frac{1}{\Gamma(a)} \sum_{k=0}^\infty
\frac{(a)_k}{k!^2}\zeta^k\left(
\frac{\Gamma'(a+k)}{\Gamma(a+k)}-\frac{2\Gamma'(k+1)}{\Gamma(k+1)}\right).
\end{equation}
 These expansions converge uniformly in $\zeta:
|\zeta|<1$ and the coefficients depend on $\epsilon$ explicitly. In
$\Omega$, the logarithmic singularities cancel each other as follows
(recall that $\zeta=iu, u>0$)
\[
-\frac{\ln
\zeta}{\Gamma(-i\epsilon/\pi)}+\frac{\ln(e^{-i\pi}\zeta)}{\Gamma(-i\epsilon/\pi)}=\frac{-i\pi}{\Gamma(-i\epsilon/\pi)}\to
0, \,\epsilon\to 0,
\]
where we accounted for the first terms in the series \eqref{ser1}
and \eqref{ser2} only since for the other terms we can use
\[
|\zeta \ln \zeta|\lesssim 1, \quad |\zeta|<1.
\]
Therefore, the required asymptotics of $\Omega$ will follow from
\[
-\lim_{\epsilon\to
0}\frac{\Gamma'(-i\epsilon/\pi)}{\Gamma^2(-i\epsilon/\pi)}=1.
\]
Now, recall that \eqref{rhp},\eqref{insideUz2},\eqref{abc} yield
\[
-\Phi_{n-1}^*(z)=O(n^{-1})\widetilde{Y}_{i,
11}^{(I)}(z)+(1+O(n^{-1})\widetilde{Y}_{i, 21}^{(I)}(z).
\]
The analysis to show that $\widetilde{Y}_{i, 11}^{(I)}(z) = O(1)$ is
nearly identical to that showing $|\widetilde{Y}_{i, 21}^{(I)}(z)|
\sim 1$ except that it may be performed with less care, since only
an upper bound is needed. The estimates we obtained prove
\eqref{uno} in Lemma \ref{l1}.\smallskip

\subsection{$\| \Phi_n^* \Psi_n^* + z \Phi_n \Psi_n\|_{L^\infty(\mathbb{T})}>_{\epsilon} \ln n$}

We will investigate  $\Phi_n^* \Psi_n^* + z \Phi_n \Psi_n$ for
$z=e^{i\theta}, \theta\in (\pi/2+n^{-0.5},\pi/2+2n^{-0.5})$.  Since
$\displaystyle \zeta = n\ln \frac{z}{i}=iu, u\sim \sqrt{n}$, this
puts us in the $\zeta \to \infty$ regime when using the parametrix
$P_{z_1}$. This also allows us to easily perform the final
multiplication in \eqref{abc}, since all elements in the
$\widetilde{Y}$ matrix are $O(1)$ when $|\zeta|>1$. Recall equation
\eqref{y_22psi_n}:
\begin{equation}\label{lala1}
2z^n Y_{22}(z) =(\cosh \epsilon)
\Psi_{n-1}^*(z)-F(z)\Phi_{n-1}^*(z).
\end{equation}
Performing the multiplication and noting the error,
\eqref{insideUz1} gives
$$Y_{22}(z) = \mathcal{D}(z)^{-1}e^{-\epsilon}c_2^{(1)}\zeta^{-i\epsilon/\pi}i^{-n} \psi_3^{(1)}
 + O\left(\frac{1}{n}\right).$$ By  \eqref{largezeta},
  $\mathcal{D}(z)^{-1}e^{-\epsilon}c_2^{(1)}\zeta^{-i\epsilon/\pi}i^{-n} \psi_3^{(1)}=O(n^{-1/2})$ and $Y_{22}(z)=O(n^{-1/2})$. Similarly,
   by equation \eqref{insideUz1},
\begin{equation}\label{largezetaphi}\Phi_{n}^* = \mathcal{D}^{-1}+ O\left(\frac{1}{\sqrt{n}}\right)\end{equation} for $\zeta=iu, u\sim \sqrt n$.
Therefore, we finally have
\begin{equation}\label{largezetapsi}\left|\Psi_{n}^*(z)\right|\sim \epsilon \ln n
\end{equation}
due to (26). Recall that $z^n \overline{\Phi_n^*} = \Phi_n$. So, we
can write
$$\Phi_n^*\Psi_n^* + z\Phi_n \Psi_n =  \Phi_n^*\Psi_n^* + z^{2n+1}\overline{\Phi_n^*\Psi_n^*} =
 \Phi_n^*\Psi_n^*\left(1+z^{2n+1}\frac{\overline{\Phi_n^*\Psi_n^*}}{\Phi_n^*\Psi_n^*}\right).$$
Inserting \eqref{lala1},\eqref{largezetaphi}, and using
$|\mathcal{D}|\sim 1$, we have
  %Since (33) $\Phi_n^*(z) = \mathcal{D}(z)^{-1} + O\left(\frac{1}{\sqrt{n}}\right)$ and (34) $\Psi_n^*(z) = F(z)\mathcal{D}(z)^{-1} + O\left(\frac{1}{\sqrt{n}}\right)$, we have
\begin{equation}\label{ll1}
\frac{\overline{\Phi_n^*\Psi_n^*}}{\Phi_n^* \Psi_n^*} =
\frac{\overline{F(z)(\mathcal{D}(z)^{-2}+
O(1/\sqrt{n}))}}{F(z)(\mathcal{D}(z)^{-2}+
O(1/\sqrt{n}))}=\frac{\overline F}{F}\cdot
\left(\frac{\overline{\mathcal{D}}}{\mathcal{D}}\right)^{-2}(1+O(n^{-0.5})).
\end{equation}
From \eqref{Herglotz} and \eqref{Szego} we can compute
\[
\frac{\overline F}{F}=-1+o(1), \quad n\to\infty
\]
and
\[
\frac{\overline{\mathcal{D}}}{\mathcal{D}}=\exp\left(\frac{2i\epsilon}{\pi}
\ln(\theta/2-\pi/4)\right)(1+o(1)), \quad n\to\infty
\]
in our range of $z=e^{i\theta}$. Therefore, substitution into
\eqref{ll1} shows that there is some $\theta_0:\theta_0\in
(\pi/2+n^{-0.5},\pi/2+2n^{-0.5})$ for which
  $$\left|1+z_0^{2n+1}\frac{\overline{\Phi_n^*(z_0)\Psi_n^*(z_0)}}{\Phi_n^*(z_0)\Psi_n^*(z_0)}\right|\sim 1$$
where $z_0=e^{i\theta_0}$.

Since, by \eqref{largezetapsi}, we have $|\Psi_n^*(z_0)| = O(\ln
n)$, and  $|\Phi_n^*|\sim 1$, we get
$$\| \Phi_n^* \Psi_n^* + z \Phi_n \Psi_n\|_{L^\infty(\mathbb{T})}>_{\epsilon} \ln n$$
as desired.\smallskip

\subsection{$\displaystyle \frac{\Psi_n^*(z)}{\Phi_n^*(z)} + \frac{\Psi_n^*(-z)}{\Phi_n^*(-z)} = O(1),z\in \mathbb{T}, n>n_0(\epsilon)$}

Outside of $U_{z_{1(2)}}$ this statement is trivial by
\eqref{outsieUzj}, so we only consider $z$ inside $U_{z_{1(2)}}$.
Further, since the calculations are exactly similar in $U_{i}$ and
$U_{-i}$, we let $z \in U_i$.

Before we proceed with the analysis of Riemann-Hilbert problem, let
us make two remarks. Firstly, since $|\Phi_n^*|\sim 1,
n>n_0(\epsilon)$, we only need to show that $U_{2n}$, defined by
$U_{2n}(z)=\Psi_n^*(z)\Phi_n^*(z)+\Psi_n^*(-z)\Phi_n^*(-z)$,
satisfies
\[
\|U_{2n}\|_{L^\infty(\mathbb{T})}\lesssim 1,\quad
\]
Secondly, $U_{2n}$ is a polynomial of degree at most $2n$ and
\[
\|U_{2n}\|_{L^\infty(\mathbb{T})}<_\epsilon n^{C_\epsilon}
\]
by \eqref{rough}.  The Bernstein inequality gives us
\[
\|U_{2n}'\|_{L^\infty(\mathbb{T})}<_\epsilon n^{1+C_\epsilon}
\]
Thus, to prove $\|U_{2n}\|_{L^\infty(\mathbb{T})}=O(1)$, we only
need
\begin{equation}\label{rough2}
|U_{2n}(e^{i\theta})|\lesssim 1
\end{equation}
for $\theta: \theta\in (\pi/2+e^{-\sqrt n}, \pi/2+\delta_1)$ and the
parameter $\delta_1$ here is of the same size as the radius of
$U_{i}$.

 Recall that $Y^{(1(2),I)}$ denotes the $Y$ matrix in the region
$I$ that corresponds to point $z_{1(2)}$, respectively. By
\eqref{y_22psi_n},
$$2z^n\frac{Y_{22}^{(1, I)}(z)}{Y_{21}^{(1, I)}(z)}+2(-z)^n\frac{Y_{22}^{(2, I)}(-z)}{Y_{21}^{(2, I)}(-z)}
= \Bigl(F(z)+F(-z)\Bigr)-(\cosh
\epsilon)\left(\frac{\Psi_{n-1}^*(z)}{\Phi_{n-1}^*(z)}+
\frac{\Psi_{n-1}^*(-z)}{\Phi_{n-1}^*(-z)}\right),$$ so we want to
show
$$\Bigl(F(z)+F(-z)\Bigr)-2\left(z^n\frac{Y_{22}^{(1, I)}(z)}{Y_{21}^{(1, I)}(z)}+(-z)^n\frac{Y_{22}^{(2, I)}(-z)}{Y_{21}^{(2, I)}(-z)}\right) = O(1)$$
uniformly on $\mathbb{T}$ provided that $n$ is large enough.

The formula \eqref{Herglotz} implies
$$F(z) + F(-z) =-\frac{i(e^\epsilon-e^{-\epsilon})}{\pi}\left(\ln\left(\frac{i-z}{i+z}\right)+
\ln\left(\frac{i+z}{i-z}\right)\right)+e^{-\epsilon}+e^\epsilon =
e^\epsilon+e^{-\epsilon}.$$ Due to this cancelation, we may focus on
$$ z^n\frac{Y_{22}^{(1)}(z)}{Y_{21}^{(1)}(z)} +(-z)^n
\frac{Y_{22}^{(2)}(-z)}{Y_{21}^{(2)}(-z)}$$ where $z=e^{i\theta},
\,\theta\in (\pi/2,\pi/2+\delta_1)$. In fact, since we know that
$|\Phi_n^*| \sim 1$ uniformly on $\mathbb{T}$, we may multiply out
the denominators and only examine
$$z^nY_{22}^{(1)}(z)Y_{21}^{(2)}(-z) + (-z)^nY_{22}^{(2)}(-z)Y_{21}^{(1)}(z).$$
We must take care in performing the final multiplication in the
Riemann-Hilbert problem. We have previously seen that $Y_{21}(z) =
\widetilde{Y}_{21}(z) + O\left(\frac{1}{n}\right) \sim O(1)$
uniformly $z \in \mathbb{T}$. Therefore,
\begin{equation}
\label{Y_eqn_p3}
z^nY_{22}^{(1)}(z)Y_{21}^{(2)}(-z) + (-z)^nY_{22}^{(2)}(-z)Y_{21}^{(1)}(z) =
\end{equation}
\[
 z^n\widetilde{Y}_{22}^{(1)}(z)\widetilde{Y}_{21}^{(2)}(-z) +
 (-z)^n\widetilde{Y}_{22}^{(2)}(-z)\widetilde{Y}_{21}^{(1)}(z)
  +O\left(\max_{j=1,
  2}\frac{\left|\widetilde{Y}_{12}^{(j)}(z)\right|+\left|\widetilde{Y}_{22}^{(j)}(z)\right|}{n}\right).
\]
Since the final term has a logarithmic singularity at $z=i$, we will
handle it away from this point. In the range $z = e^{i\theta},
\pi/2+e^{-\sqrt n}<\theta<\pi/2+\delta_1$, we
 have (by \eqref{insideUz1} and estimates on $\psi$)
$$\max_{j=1, 2} \frac{|\widetilde Y_{12}^{(j)}(z)|+|\widetilde Y_{12}^{(j)}(z)|}{n} = O\left(\frac{|\ln\zeta|}{n}\right) =O(\sqrt{n}/n)=O(n^{-1/2}).$$
 So,
it suffices to show
\begin{equation}
\label{reduced_p3_OOI}
 z^n\widetilde{Y}_{22}^{(1)}(z)\widetilde{Y}_{21}^{(2)}(-z) + (-z)^n\widetilde{Y}_{22}^{(2)}(-z)\widetilde{Y}_{21}^{(1)}(z)  =
 O(1).
\end{equation}
By \eqref{insideUz1} and \eqref{insideUz2}, we have
$$z^n\widetilde{Y}_{22}^{(1)}(z)\widetilde{Y}_{21}^{(2)}(-z) + (-z)^n\widetilde{Y}_{22}^{(2)}(-z)\widetilde{Y}_{21}^{(1)}(z) =
\frac{1}{e^\epsilon \mathcal{D}(z)\zeta^{i\epsilon/\pi}} \left(\frac{z}{i}\right)^n c_2^{(1)}\psi_3^{(1)} \left(\frac{-\zeta^{i\epsilon/\pi}}{\mathcal{D}(-z)}\left(\psi_1^{(2)}- \left(\frac{-z}{-i}\right)^nc_2^{(2)}\psi_3^{(2)}\right)\right)$$
$$+\frac{e^\epsilon \zeta^{i\epsilon/\pi}}{\mathcal{D}(-z)} \left(\frac{-z}{-i}\right)^n c_2^{(2)}
\psi_3^{(2)}\left(\frac{-1}{\mathcal{D}(z)\zeta^{i\epsilon/\pi}}\left(\psi_1^{(1)}
- \left(\frac{z}{i}\right)^nc_2^{(1)}\psi_3^{(1)}\right)\right). $$
Notice it is not ambiguous to leave the arguments of the $\psi$
functions unidentified, as $\displaystyle \zeta = n\ln
\frac{z}{z_j}$ and $\displaystyle \frac{z}{i} = \frac{-z}{-i}$.
Taking absolute values, we have
$$\left|z^n\widetilde{Y}_{22}^{(1)}(z)\widetilde{Y}_{21}^{(2)}(-z) + (-z)^n\widetilde{Y}_{22}^{(2)}(-z)\widetilde{Y}_{21}^{(1)}(z)\right| =$$
$$ \left|\mathcal{D}(z)\mathcal{D}(-z)\right|^{-1}\left|e^{-\epsilon}\left(c_2^{(1)}\psi_3^{(1)}\psi_1^{(2)}-\left(\frac{z}{i}\right)^n
c_2^{(1)}c_2^{(2)}
\psi_3^{(2)}\psi_3^{(1)}\right)+e^\epsilon\left(c_2^{(2)}\psi_1^{(1)}\psi_3^{(2)}-\left(\frac{z}{i}\right)^nc_2^{(1)}
c_2^{(2)}\psi_3^{(2)}\psi_3^{(1)}\right)\right|$$

$$\lesssim \bigg|e^{-\epsilon} \frac{\Gamma(1+\frac{i\epsilon}{\pi})}{\Gamma(-\frac{i\epsilon}{\pi})}
\psi\left(1+\frac{i\epsilon}{\pi}, 1, e^{-i\pi} \zeta\right)\psi\left(\frac{i\epsilon}{\pi}, 1,
\zeta\right)+e^{\epsilon} \frac{\Gamma(1-\frac{i\epsilon}{\pi})}{\Gamma(\frac{i\epsilon}{\pi})}
 \psi\left(1-\frac{i\epsilon}{\pi}, 1, e^{-i\pi} \zeta\right)\psi\left(-\frac{i\epsilon}{\pi}, 1, \zeta\right)$$
$$-(e^\epsilon+e^{-\epsilon}) \frac{\Gamma(1+\frac{i\epsilon}{\pi})}{\Gamma(-\frac{i\epsilon}{\pi})}
\frac{\Gamma(1-\frac{i\epsilon}{\pi})}{\Gamma(\frac{i\epsilon}{\pi})}\left(\frac{z}{i}\right)^n\psi
\left(1+\frac{i\epsilon}{\pi}, 1, e^{-i\pi}
\zeta\right)\psi\left(1-\frac{i\epsilon}{\pi}, 1,
e^{-i\pi}\zeta\right)\bigg|.$$ By
\eqref{largezeta},\eqref{ss1},\eqref{ss2}, these expressions are
uniformly bounded in $\zeta,
|\zeta|>1,\epsilon<\epsilon_0,n>n_0(\epsilon)$.  Thus, we only need
to consider the case $\zeta: |\zeta|<1$.   On that interval,
$(z/i)^n=e^\zeta=1+O(\zeta)$.   We are concerned with the
logarithmic singularities and the constant terms in the series
expansions for $\psi$.    We isolate these terms and denote their
sum $c_0$.  Using the notation $\displaystyle d(z) =
\frac{\Gamma'(z)}{\Gamma(z)}$ for the digamma function, we get
$$c_0=e^{-\epsilon} \frac{1}{\Gamma(-\frac{i\epsilon}{\pi})\Gamma(\frac{i\epsilon}{\pi})}
\left(\ln(e^{-i\pi}\zeta) + d\left(1+\frac{i\epsilon}{\pi}\right) - 2d(1)\right)\left(\ln \zeta + d\left(\frac{i\epsilon}{\pi}\right)-2d(1)\right)$$
$$-(e^{\epsilon}+e^{-\epsilon})  \frac{1}{\Gamma(\frac{i\epsilon}{\pi})
\Gamma(-\frac{i\epsilon}{\pi})}\left(\ln(e^{-i\pi}\zeta) + d\left(1+\frac{i\epsilon}{\pi}\right)
 - 2d(1)\right)\left(\ln(e^{-i\pi}\zeta) + d\left(1-\frac{i\epsilon}{\pi}\right)-2d(1)\right)
$$
$$+ e^\epsilon\frac{1}{\Gamma(\frac{i\epsilon}{\pi})\Gamma(-\frac{i\epsilon}{\pi})}\left(\ln(e^{-i\pi}\zeta) +
 d\left(1-\frac{i\epsilon}{\pi}\right) - 2d(1)\right)\left(\ln \zeta + d\left(\frac{-i\epsilon}{\pi}\right)-2d(1)\right).$$
Performing these multiplications, writing $\ln (iu)=\ln u+i\pi/2,
\,\ln (e^{-i\pi}iu)=\ln u-i\pi/2$, and pulling out the common factor
yields
$$c_0=\frac{\ln u}{\Gamma(\frac{i\epsilon}{\pi})\Gamma(-\frac{i\epsilon}{\pi})}\left(e^{-\epsilon}
\left(d\left(\frac{i\epsilon}{\pi}\right) -
d\left(1-\frac{i\epsilon}{\pi}\right)+i\pi\right)+e^\epsilon\left(d\left(\frac{-i\epsilon}{\pi}\right)+
d\left(1+\frac{i\epsilon}{\pi}\right)+i\pi\right)\right)+O(1),$$
where $\zeta=iu$. Using the reflection formula \eqref{reflection}
gives
$$\frac{\ln u}{\Gamma(\frac{i\epsilon}{\pi})\Gamma(-\frac{i\epsilon}{\pi})}(e^\epsilon(-\pi\cot(-i\epsilon))+e^\epsilon i \pi +
 e^{-\epsilon}(-\pi\cot(i\epsilon))+e^{-\epsilon}i\pi)$$
$$=\frac{\pi\ln u}{\Gamma(\frac{i\epsilon}{\pi})\Gamma(-\frac{i\epsilon}{\pi})}(i(e^\epsilon+e^{-\epsilon})
- e^\epsilon\cot(-i\epsilon)-e^{-\epsilon}\cot(i\epsilon))=0,$$
because $\displaystyle \cot z =
\frac{i(e^{iz}+e^{-iz})}{e^{iz}-e^{-iz}}$. Therefore,
$$z^n\widetilde{Y}_{22}^{(1)}(z)\widetilde{Y}_{21}^{(2)}(-z) + (-z)^n\widetilde{Y}_{22}^{(2)}(-z)\widetilde{Y}_{21}^{(1)}(z)  = O(1)$$
and
$$\frac{\Psi_{n}^*(z)}{\Phi_n^*(z)} + \frac{\Psi_n^*(-z)}{\Phi_n^*(-z)} = O(1)$$
uniformly in $z \in \mathbb{T}$ for $n$ large enough. This finishes
the proof of Lemma \ref{l1}.\end{proof}

{\Large \part*{Acknowledgement}} The work of SD done in the first
part of the paper was supported by RSF-14-21-00025 and his research
on the rest of the paper was supported by the grant NSF-DMS-1464479.
The research of KR was supported by the RTG grant NSF-DMS-1147523.

\end{document}